\documentclass[11pt,article]{amsart}
\usepackage{amssymb,url}
\usepackage{graphicx}

\theoremstyle{plain}
\newtheorem{theorem}{Theorem}
\newtheorem*{Theorem}{Theorem}

\newtheorem{corollary}{Corollary}

\newcommand\ran{\operatorname{ran}}
\newcommand\tr{\operatorname{Tr}}

\newcommand\Area{\operatorname{\blacklozenge}}
\newcommand\Red{\operatorname{Red}}
\newcommand\RP{\operatorname{RP}}
\newcommand\RPS{\operatorname{RP\mathbb{S}}}

\newcommand{\C}{\mathbb{C}}

\newcommand{\R}{\mathbb{R}}

\newcommand{\Sphere}{\mathbb{S}}

\newcommand{\irB}{\mathcal{B}}

\newcommand{\irH}{\mathcal{H}}

\newcommand{\irM}{\mathcal{M}}

\newcommand{\irP}{\mathcal{P}}

\begin{document}

\title[]{Maps on real Hilbert spaces preserving the area of parallelograms and a preserver problem on self-adjoint operators}
\author{Gy\"orgy P\'al Geh\'er}
\address{Bolyai Institute\\
University of Szeged\\
H-6720 Szeged, Aradi v\'ertan\'uk tere 1, Hungary}
\address{MTA-DE "Lend\"ulet" Functional Analysis Research Group, Institute of Mathematics\\
University of Debrecen\\
H-4010 Debrecen, P.O.~Box 12, Hungary}
\email{gehergy@math.u-szeged.hu}
\urladdr{\url{http://www.math.u-szeged.hu/~gehergy/}}

\begin{abstract}
In this paper first we describe all (not necessarily linear or bijective) transformations on $\R^d$ with $2\leq d<\infty$ which preserve the area of parallelograms spanned by any two vectors. We also characterize those (not necessarily linear) bijections on an arbitrary real Hilbert space that preserve the latter quantity. This answers a question raised by Rassias and Wagner, and it can be considered as a variant of the famous Wigner theorem on real Hilbert spaces which plays an important role in quantum mechanics. As a consequence, we solve a preserver problem of Moln\'ar and Timmermann which has remained open stubbornly only in the two-dimensional case. Finally this two-dimensional result will be applied in order to strengthen their theorem in higher dimensions.
\end{abstract}

\subjclass[2010]{Primary: 46C05, 51M05, 51M25, 47B49. Secondary: 47N50, 51P05, 81Q99. }
\keywords{Preserver problems, area of parallelogram, commutativity of self-adjoint operators, unitarily invariant norm.}

\maketitle



\section{Introduction}

The characterization of geometric transformations under mild assumptions is a modern direction of geometry. A very nice example of it is a theorem of J.~Lester. He proved in \cite{Le} that any mapping $\phi$ of $\R^d$ into itself such that $\phi(\vec{a}), \phi(\vec{b})$ and $\phi(\vec{c})$ are the vertices of a triangle of area 1 whenever $\vec{a}, \vec{b}$ and $\vec{c}$ are the vertices of a triangle of area 1, has to be an isometry, i.~e.~a composition of a linear orthogonal operator and a translation. W.~Huang proved a similar result concerning lines instead of points (see \cite{Hu}). W. Benz's book (\cite{Be}) contains many theorems which are of a similar spirit, in particular it contains the above mentioned two results as well.

Let $E$ be a real (not necessarily separable) Hilbert space. Whenever $E$ is $d$-dimensional with $2\leq d<\infty$ we will often identify it with $\R^d$, and linear operators on $\R^d$ with $d\times d$ real matrices in the natural way (i.~e.~written in the standard orthonormal base). Any two points/vectors $\vec{a},\vec{b}\in E$ span the parallelogram $\{ s\vec{a}+t\vec{b} \colon s,t\in[0,1] \}$, and the area of this parallelogram is defined by the following usual formula:
\begin{equation}\label{E:Pardef}
\Area(\vec{a},\vec{b}) = \sqrt{|\vec{a}|^2\cdot|\vec{b}|^2-\langle\vec{a},\vec{b}\rangle^2} = \sqrt{|\vec{a}|^2\cdot|\vec{b}|^2-\frac{1}{4}(|\vec{a}-\vec{b}|^2-|\vec{a}|^2-|\vec{b}|^2)^2}
\end{equation}
where $|\cdot|$ denotes the norm on $E$. In \cite{RW} T.~M.~Rassias and P.~Wagner, inspired by \cite{Be}, posed the following problem: describe all mappings (linear or not) of a real Hilbert space $E$ into itself which preserve the areas of parallelograms. Let us point out that in Lester's theorem the area of the triangle with vertices $\vec{a}, \vec{b}$ and $\vec{c}$ is exactly $\frac{1}{2}\cdot\Area(\vec{a}-\vec{c}, \vec{b}-\vec{c})$. This problem remained open. In the first part of the present paper, we will solve it on $\R^d$ for general transformations, and the infinite dimensional analogue for bijections. Namely, we will prove the following theorem.

\begin{theorem}\label{T:RW}
Let $E$ be a real (not necessarily separable) Hilbert space and $\phi\colon E\to E$ be a transformation such that
\begin{equation}\label{E:Parprop}
\Area(\vec{a},\vec{b}) = \Area(\phi(\vec{a}),\phi(\vec{b})) \qquad (\forall \; \vec{a}, \vec{b} \in E).
\end{equation}
\begin{itemize}
\item[(i)] If $\dim E = 2$, then there exists a function $\epsilon\colon E\to\{-1,1\}$ and a linear operator $A\colon E\to E$ with $|\det A| = 1$ such that the following holds:
\begin{equation}\label{E:RW2d}
\phi(\vec{a}) = \epsilon(\vec{a})A\vec{a} \qquad (\vec{a} \in E).
\end{equation}
\item[(ii)] If $2 < \dim E < \infty$, then there exists a function $\epsilon\colon E\to\{-1,1\}$ and an orthogonal linear operator $R\colon E\to E$ such that
\[ 
\phi(\vec{a}) = \epsilon(\vec{a})R\vec{a} \qquad (\vec{a} \in E) 
\]
is satisfied.
\item[(iii)] If $\dim E = \infty$ and in addition $\phi$ is assumed to be bijective, then there exists a function $\epsilon\colon E\to\{-1,1\}$ and a linear, surjective isometry $R\colon E\to E$ such that we have
\[ 
\phi(\vec{a}) = \epsilon(\vec{a})R\vec{a} \qquad (\vec{a} \in E).
\]
\end{itemize}
\end{theorem}

We recall the well-known fact
\[ \Area(A\vec{a},A\vec{b}) = |\det A|\Area(\vec{a},\vec{b}) \qquad (\forall \; \vec{a}, \vec{b} \in\R^2) \]
where $A\colon\R^2\to\R^2$ is a linear operator. It is also quite easy to see, with the help of the singular-value decomposition for real matrices, that any linear operator $A\colon\R^d\to\R^d$ where $d>2$ preserves the area of parallelograms exactly when all singular values are 1, or equivalently if $A$ is orthogonal. We note that in the proof of the above theorem the latter fact will be not used. We also point out that the above theorem can be considered as an analogue of the famous Wigner theorem on real Hilbert spaces (see \cite{MP,Ra}). Namely, Wigner's theorem characterizes all transformations of a Hilbert space that preserves the absolute value of the inner product. However here, if $d=3$, Theorem \ref{T:RW} characterizes those transformations of $\R^3$ into itself such that it preserves the norm of the cross (or vectorial) product. This is a significant quantity in physics. In fact, in the proof of (ii)-(iii) of Theorem \ref{T:RW}, we will indeed reduce the problem to the real version of Wigner's theorem. There are several proofs for Wigner's theorem on complex Hilbert spaces e.~g.~\cite{Ch,Ge,Gy,Mo,Mo2}, and many of them works also for the real case.

In the second part of the paper we will consider a complex and separable Hilbert space $\irH$. The set of bounded and linear operators on $\irH$ will be denoted by $\irB(\irH)$. The symbol $\irB_s(\irH)$ will stand for the real vector-space of bounded, self-adjoint operators. Whenever we consider a finite dimensional complex Hilbert space, we will usually identify it with $\C^d$, and the elements of $\irB(\C^d)$ with $d\times d$ complex matrices in the natural way. The commutator of two elements $A,B\in \irB(\irH)$ is defined by $[A,B] := AB-BA$. The usual vector- and operator-norm will be denoted by $\|\cdot\|$. A conjugate-linear operator $U$ on $\irH$ is said to be antiunitary if $U^*U = UU^* = I$ holds. 

The general structure of commutativity preserving maps on $\irB_s(\irH)$ was described in \cite{MS,Se}. Namely, if $2<d<\infty$, then any (not necessarily bijective or linear) transformation $\phi$ of $\irB_s(\C^d)$ that preserves commutativity in both directions sends each element $A\in \irB_s(\C^d)$ -- up to unitary or antiunitary equivalence -- into some polynomial of it: $p_A(A)$ where $p_A$ is injective on the spectrum of $A$. If $\dim\irH = \aleph_0$ and in addition $\phi$ is assumed to be bijective, then a similar conclusion holds with some bounded Borel functions $f_A$. The relation of commutativity between self-adjoint operators is very important in quantum physics, since it represents the compatibility of the corresponding observables (i.~e.~if they can be measured simultaneously in every state of the quantum system). In a two-dimensional space the corresponding problem is very easy. In fact two matrices $A,B \in \irB_s(\C^2)$ commute exactly when $\alpha A + \beta B \in \{0, I\}$ holds with some real numbers $\alpha,\beta\in\R$. Therefore there are many transformations of $\irB_s(\C^2)$ into itself which preserve the relation of commutativity in both directions. 

Naturally, if we pose a stronger condition, we may obtain more regular forms. One reasonable quantity which represents a measure of commutativity (or compatibility) is the norm of the commutator. In \cite{MT} L.~Moln\'ar and W.~Timmermann proved the following result concerning bijective transformations.

\begin{Theorem}[L.~Moln\'{a}r and W.~Timmermann, \cite{MT}, 2011]\label{T:MTopnorm}
Let $\irH$ be a complex separable Hilbert space with $\dim\irH > 2$. Assume $\phi\colon \irB_s(\irH) \to \irB_s(\irH)$ is a bijection such that
\[ 
\big\|[\phi(A),\phi(B)]\big\| = \big\|[A,B]\big\|\quad(A,B \in \irB_s(\irH)).
\]
Then there exist either a unitary or an antiunitary operator $U$ on $\irH$ and functions $f\colon \irB_s(\irH) \to \R$, $\tau\colon \irB_s(\irH)\to \{-1, 1\}$ such that
\[ 
\phi(A) = \tau(A)UAU^* + f(A)I\quad(A \in \irB_s(\irH)).
\]
\end{Theorem}

At the end of their paper Moln\'ar and Timmermann pointed out that their technique cannot be applied in two dimensions and that they do not know whether the same conclusion is true in that case. The linear version of this two-dimensional problem was solved recently in \cite{GN} by the author of the present paper and G.~Nagy. In fact, via that technique we were able to describe those, not necessarily bijective, linear transformations in finite dimensions that preserve a given unitarily invariant norm of the commutator. A norm $|||\cdot|||$ on $\irB(\irH)$ is said to be unitarily invariant if $|||UAV||| = |||A|||$ is valid for every $A,U,V\in\irB(\irH)$ where $U$ and $V$ are unitary. The operator norm is a trivial example. A characterization of unitarily invariant norms on $\C^{d\times d}$ can be found in \cite[Section IV.2.]{Bha}. However, the general two-dimensional problem remained persistently open. Here we will solve it, moreover we will improve the above Moln\'ar-Timmermann theorem in the following two ways in finite dimensions: we do not assume bijectivity and we replace the operator norm with general unitarily invariant norms.

\begin{theorem}\label{T:MTgen}
Fix a unitarily invariant norm $|||\cdot|||$ on $\C^{d\times d}$ where $d\geq 2$. Let $\phi\colon \irB_s(\C^d)\to \irB_s(\C^d)$ be a (not necessarily bijective or linear) transformation for which the following holds:
\begin{equation}\label{E:PMTUI}
|||[A,B]||| = |||[\phi(A),\phi(B)]||| \qquad (A, B \in \irB_s(\C^d)).
\end{equation}
Then there exist functions $\tau\colon \irB_s(\C^d)\to\{-1,1\}$, $f\colon \irB_s(\C^d)\to\R$ and a unitary or antiunitary operator $U$ such that
\[ 
\phi(A) = \tau(A)UAU^* + f(A)I \qquad (A \in \irB_s(\C^d))
\]
is satisfied.
\end{theorem}

In the separable and  infinite dimensional case we can also strengthen the Moln\'ar-Timmermann theorem but in this case bijectivity is crucial.

\begin{theorem}\label{T:MTgeninfty}
Let $\irH$ be a separable Hilbert space and fix a unitarily invariant norm $|||\cdot|||$ on $\irB(\irH)$. Let $\phi\colon \irB_s(\irH)\to \irB_s(\irH)$ be a bijection for which the following holds:
\begin{equation}\label{E:PMTUIinfty}
|||[A,B]||| = |||[\phi(A),\phi(B)]||| \qquad (A, B \in H_d).
\end{equation}
Then there exist functions $\tau\colon \irB_s(\irH)\to\{-1,1\}$, $f\colon \irB_s(\irH)\to\R$ and a unitary or antiunitary operator $U$ such that
\[ 
\phi(A) = \tau(A)UAU^* + f(A)I \qquad (A \in \irB_s(\irH))
\]
is satisfied.
\end{theorem}

In order to prove Theorem \ref{T:MTgen} and \ref{T:MTgeninfty} first we will reduce the two-dimensional version of Theorem \ref{T:MTgen} to the three-dimensional version of Theorem \ref{T:RW}. This will show how strongly connected our two results are. Then we will use the two-dimensional case in order to finish our proof in the general case. It is important to note that in at least three dimensions Theorem \ref{T:MTgen} and \ref{T:MTgeninfty} could be proven in the way as in \cite{MT} using \cite[Theorem 1.2]{Se} and \cite[Corollary 2]{MS}. However, this was not pointed out in \cite{MT}, thus for the sake of completeness we present its proof but with another method.

We will give the proofs of our results in the next section and we will close our paper with some discussing and posing some open problems.


\section{Proofs}

We begin with the proof of Theorem \ref{T:RW}.

\begin{proof}[Proof of Theorem \ref{T:RW}]
It is easy to see that the images of any two vectors are linearly dependent if and only if the vectors were originally linearly dependent. Therefore $\phi(\vec{v}) = \vec{0}$ holds exactly when $\vec{v} = \vec{0}$. It is also quite trivial that we have 
\begin{equation}\label{E:quasihom}
\phi(t\cdot\vec{v}) \in \{\pm t\cdot\phi(\vec{v})\} \quad (\vec{v}\in E, t\in\R),
\end{equation}
since $\Area(\phi(t\cdot\vec{v}),\phi(\vec{w})) = \Area(s \cdot\phi(\vec{v}),\phi(\vec{w}))$ holds if and only if $s=\pm t$ whenever $\vec{v}$ and $\vec{w}$ are assumed to be linearly independent.
Let $\{\vec{e}_j\colon j\in J\}$ be an orthonormal base in $E$. Whenever $E$ is finite dimensional we implicitly assume that $J = \{1,2,\dots d\}$ where $d = \dim E$.

Let us consider a non-zero vector $\vec{v} \in E$ and a sequence $\{\vec{v}_n\}_{n=1}^\infty$ such that $\lim_{n\to\infty}|\vec{v}_n-\vec{v}| = 0$. Let $P\colon E\to E$ denote the orthogonal projection with precise range $\R\cdot\phi(\vec{v})$. Obviously we have
\begin{equation}\label{E:I-Pv_n}
\begin{gathered}
\lim_{n\to\infty} |(I-P)\phi(\vec{v}_n)| = \lim_{n\to\infty} \frac{\Area\big(\phi(\vec{v}),(I-P)\phi(\vec{v}_n)\big)}{|\phi(\vec{v})|} \\
= \lim_{n\to\infty} \frac{\Area(\phi(\vec{v}),\phi(\vec{v}_n))}{|\phi(\vec{v})|} = \lim_{n\to\infty} \frac{\Area(\vec{v},\vec{v}_n)}{|\phi(\vec{v})|} = 0. 
\end{gathered}
\end{equation}
Since there exists an index $j_0$ such that $\vec{v}$ and $\vec{e}_{j_0}$ are linearly independent, $\phi(\vec{v})$ and $\phi(\vec{e}_{j_0})$ are linearly independent as well. Therefore if $\{P\phi(\vec{v}_n)\}_{n=1}^\infty$ or equivalently $\{\phi(\vec{v}_n)\}_{n=1}^\infty$ was unbounded, $\{\Area(\phi(\vec{e}_{j_0}),\phi(\vec{v}_n))\}_{n=1}^\infty$ would be unbounded as well, which is impossible since it is convergent. Hence by \eqref{E:I-Pv_n} we see that any subsequence of $\{\phi(\vec{v}_n)\}_{n=1}^\infty$ clusters to at least one point and any such cluster point is of the form $t\cdot\phi(\vec{v})$ with a number $t\in\R$. Let $\{\phi(\vec{v}_{n_k})\}_{k=1}^\infty$ be such a subsequence that converges to $t\cdot\phi(\vec{v})$. The continuity of $\Area(\cdot,\cdot)$ implies 
\[ 
|t|\cdot\Area(\vec{v},\vec{e}_{j_0}) = \Area(t\cdot\phi(\vec{v}),\phi(\vec{e}_{j_0})) = \lim_{k\to\infty} \Area(\phi(\vec{v}_{n_k}),\phi(\vec{e}_{j_0})) 
\]
\[ 
= \lim_{k\to\infty} \Area(\vec{v}_{n_k},\vec{e}_{j_0}) = \Area(\vec{v},\vec{e}_{j_0}) \neq 0, 
\]
and therefore we immediately obtain that
\begin{equation}\label{E:phiquasicont}
\phi(\vec{v}_{n_k}) \longrightarrow \pm\phi(\vec{v}) \quad (k\to\infty).
\end{equation}

Let $\Sphere$ be the set of unit vectors in $E$ and $\RPS$ be the projectivised space which is obtained by glueing together antipodal points. Let $p\colon \Sphere \to \RPS$ denote the usual covering map which sends antipodal points into one point. In $\R^d$ this gives us the real projective space denoted by $\RP^{d-1}$ in most cases (if $d=2$, $\RP^1$ is homeomorphic to a circle). By the observations made so far we easily conclude that the function
\begin{equation}\label{E:gdef}
g_\phi\colon \RPS \to \RPS, \quad g_\phi(p(\vec{v})) := p\left(\frac{1}{|\phi(\vec{v})|}\cdot\phi(\vec{v})\right) \quad (\vec{v}\in\Sphere)
\end{equation}
is well-defined, continuous and injective.

Next, we show that $g_\phi$ is a homeomorphism. First, we consider the finite dimensional case. By the domain invariance theorem for $(d-1)$-manifolds (see e.~g.~\cite{Fu}) we conclude that $g_\phi(\RP^{d-1})$ is open in $\RP^{d-1}$. However, since $\RP^{d-1}$ was compact, $g_\phi(\RP^{d-1})$ is compact as well. It follows immediately that $g_\phi$ is bijective. Since $g_\phi$ is a continuous bijection of a compact set onto itself, it has to be a homeomorphism. Second, we consider the case when $\dim E = \infty$ and $\phi$ is bijective. An easy observation verifies that $g_\phi$ is bijective as well. Since $\phi^{-1}$ also satisfies \eqref{E:Parprop}, we can define $g_{\phi^{-1}}$ in the same way as $g_\phi$, and $g_{\phi^{-1}}$ is also bijective. We will show that 
\begin{equation}\label{E:ginv}
g_{\phi}^{-1} = g_{\phi^{-1}}
\end{equation}
holds in every dimension (even if $2\leq \dim E<\infty$ which will be used later), which in particular verifies that $g_{\phi}^{-1}$ is continuous and therefore $g_{\phi}$ is a homeomorphism also when $\dim E = \infty$. In order to verify \eqref{E:ginv} we write the following where we use \eqref{E:quasihom}, and the notation $\vec{v} = \frac{1}{|\phi^{-1}(\vec{w})|}\cdot\phi^{-1}(\vec{w}) \in \Sphere$:
\[
g_{\phi^{-1}}^{-1}(p(\vec{v})) = p(\vec{w}) = p\left(|\phi^{-1}(\vec{w})|\cdot\phi(\vec{v})\right) \quad (\vec{w}\in\Sphere),
\]
and since $|\phi(\vec{v})| = \frac{1}{|\phi^{-1}(\vec{w})|}\cdot|\vec{w}| = \frac{1}{|\phi^{-1}(\vec{w})|}$, we obtain
\[
g_{\phi^{-1}}^{-1}(p(\vec{v})) = p\left(\frac{1}{|\phi(\vec{v})|}\cdot\phi(\vec{v})\right) = g_\phi(p(\vec{v})) \quad (\vec{v}\in\Sphere),
\]
which verifies \eqref{E:ginv}. Therefore $g_\phi$ is indeed a homoemorphism.

Now, we are in a position to prove (i). Let us consider the linear operator $B\colon \R^2\to\R^2$ such that $B\vec{e}_j = \phi(\vec{e}_j)$ ($j = 1,2$). The linear operator $B$ is obviously non-singular with $|\det B| = 1$, since it preserves the area of the specific parallelogram spanned by $\vec{e}_1$ and $\vec{e}_2$. We define the following function:
\[ 
\psi\colon \R^2\to\R^2, \quad \psi(\vec{v}) = B^{-1}\phi(\vec{v}). 
\]
Trivially, we have $\psi(\vec{e}_j) = \vec{e}_j$ ($j=1,2$), $\psi$ satisfies \eqref{E:Parprop} and thus
\[ 
\Area(\vec{e}_j,\psi(\vec{v})) = \Area(\vec{e}_j,\vec{v}) \quad (\vec{v}\in\R^2). 
\]
This obviously implies the following:
\[ 
\psi(x,y) \in \{(x,y),(-x,y),(x,-y),(-x,-y)\} \quad ((x,y)\in\R^2).
\]
By the continuity of $g_\psi$ we conclude that 
\[ 
\psi(\cos\varphi,\sin\varphi) \in \{\pm(\cos\varphi,\sin\varphi)\} \quad (\varphi\in]0,\pi/2[\cup]\pi,3\pi/2[) 
\]
or
\[ 
\psi(\cos\varphi,\sin\varphi) \in \{\pm(-\cos\varphi,\sin\varphi)\} \quad (\varphi\in]0,\pi/2[\cup]\pi,3\pi/2[)
\]
is satisfied, and a similar relation is fulfilled whenever $\varphi\in]\pi/2,\pi[\cup]3\pi/2,2\pi[$. Applying the continuity of $g_\psi$ and considering $\Area(\frac{1}{\sqrt{2}}(\vec{e}_1+\vec{e}_2),\frac{1}{\sqrt{2}}(\vec{e}_1-\vec{e}_2)) = \Area(\psi(\frac{1}{\sqrt{2}}(\vec{e}_1+\vec{e}_2)),\psi(\frac{1}{\sqrt{2}}(\vec{e}_1-\vec{e}_2)))$, we infer that one of the above relations holds for every $\varphi\in[0,2\pi[$. Thus applying \eqref{E:quasihom} we get that \eqref{E:RW2d} holds.

So from now on we may assume that $\dim E \geq 3$. Let us consider two non-zero and linearly independent vectors $\vec{a},\vec{b}\in E$ and let 
\[ 
C_{\vec{a},\vec{b}} := \left\{\vec{v}\in E\setminus\{0\}\colon \Area(\vec{v},\vec{a}) = \Area(\vec{v},\vec{b})\right\} \subseteq E 
\]
and
\[ 
PC_{\vec{a},\vec{b}} := \left\{p\left(\frac{1}{|\vec{v}|}\vec{v}\right)\in\RPS\colon \vec{v}\in C_{\vec{a},\vec{b}}\right\} 
\]
\[
= \left\{p(\vec{v})\in\RPS\colon |\vec{v}|=1, \Area(\vec{v},\vec{a}) = \Area(\vec{v},\vec{b})\right\} \subseteq \RPS. 
\]
It is quite easy to see that $C_{\vec{a},\vec{b}} = \lambda\cdot C_{\vec{a},\vec{b}}$ is valid for every $\lambda \neq 0$. In fact, $C_{\vec{a},\vec{b}}$ is a hyperplane whenever $|\vec{a}|=|\vec{b}|$. In the forthcoming two paragraphs we will show that $PC_{\vec{a},\vec{b}}$ contains a loop $\gamma\colon[0,1]\to PC_{\vec{a},\vec{b}}$ not homotopic to the trivial loop $\delta\colon[0,1]\to PC_{\vec{a},\vec{b}}$, $\delta\equiv\gamma(0)$ if and only if we have $|\vec{a}| = |\vec{b}|$ (concerning homotopies see \cite{Fu}). 

First, we consider the case when $|\vec{a}| < |\vec{b}|$ holds. Let $\vec{c}\neq\vec{0}$ be such a vector which is in the subspace generated by $\vec{a},\vec{b}$ and which is orthogonal to $\vec{a}$. We consider the hyperplane
\[ 
F := \{\vec{v}\colon\langle\vec{v},\vec{c}\rangle = 0\} \subseteq E. 
\]
Obviously we have $\vec{a}\in F$. Set $\vec{0}\neq\vec{d}\in F, \langle\vec{d},\vec{a}\rangle = 0$ and let $\vec{0}\neq\vec{v} = \lambda \vec{a} + \mu\vec{d}\in F$ with some $\lambda,\mu\in\R$. Using the linear independence of $\vec{a}$ and $\vec{b}$, we obtain
\[ 
\Area(\vec{v},\vec{b}) = \sqrt{|\vec{v}|^2|\vec{b}|^2-\langle\vec{v},\vec{b}\rangle^2} = \sqrt{(\lambda^2|\vec{a}|^2+\mu^2|\vec{d}|^2)|\vec{b}|^2-\lambda^2\langle\vec{a},\vec{b}\rangle^2} 
\]
\[ 
\geq \sqrt{(\lambda^2|\vec{a}|^2+\mu^2|\vec{d}|^2)|\vec{b}|^2-\lambda^2|\vec{a}|^2|\vec{b}|^2} = |\mu|\cdot|\vec{d}|\cdot|\vec{b}| \geq |\mu|\cdot|\vec{d}|\cdot|\vec{a}| = \Area(\vec{v},\vec{a})
\]
where the first inequality is strict whenever $\lambda\neq 0$, and the second one cannot be an equation unless $\mu = 0$. Hence we get $\Area(\vec{v},\vec{b}) > \Area(\vec{v},\vec{a})$. We immediately conclude that $F$ has to be disjoint from $C_{\vec{a},\vec{b}}\subseteq E$. Therefore if we consider $PC_{\vec{a},\vec{b}}\subseteq\RPS$ it will be contained in a subset of $\RPS$ that is homeomorphic to an open half-sphere. However, it is simply connected, since it is homeomorphic to the intersection of the unit open ball and a hyperplane (simply consider the orthogonal projection onto $F$ which gives rise to a homeomorphism between these two sets). Hence in this set, every loop $\gamma$ is obviously homotopic to the trivial loop $\delta\equiv\gamma(0)$.

Second, let us assume that $|\vec{a}| = |\vec{b}|$ holds. Let $K$ be a three-dimensional subspace which contains $\vec{a}$ and $\vec{b}$, and let $\vec{d}\in K$ be a unit vector which is orthogonal to both $\vec{a}$ and $\vec{b}$. We consider the parametrization of a half-circle with centre $\vec{0}$
\[
\tilde\gamma\colon[0,1]\to K
\] 
such that $\tilde\gamma(0) = \vec{d}$, $\tilde\gamma(1/2) = \frac{1}{|\vec{a}+\vec{b}|}\cdot(\vec{a}+\vec{b})$ and $\tilde\gamma(1) = -\vec{d}$. Trivially we have $\tilde\gamma([0,1])\subseteq C_{\vec{a},\vec{b}}$. The curve $\gamma = p\circ\tilde{\gamma}$ is clearly a loop in $\RPS$ such that the lifted curve in $\Sphere$ with beginning point $\vec{d}$ is exactly $\tilde{\gamma}$. Because of the homotopy lifting lemma $\gamma$ cannot be homotopic to the trivial loop $\delta\equiv\gamma(0)$.

Now, let us suppose that $0 < |\vec{a}| = |\vec{b}|$ is satisfied. Then clearly by \eqref{E:gdef} we have $g_\phi(PC_{\vec{a},\vec{b}}) \subseteq PC_{\phi(\vec{a}),\phi(\vec{b})}$, and since $g_\phi$ is a homeomorphism and $g_{\phi^{-1}} = g_{\phi}^{-1}$, we obtain
\[
g_\phi(PC_{\vec{a},\vec{b}}) = PC_{\phi(\vec{a}),\phi(\vec{b})}.
\]
By the above observations, $PC_{\vec{a},\vec{b}}$ contains a loop $\gamma$ which is not homotopic to the trivial loop $\delta\equiv\gamma(0)$. Since $g_\phi$ is a homeomorphism, $PC_{\phi(\vec{a}),\phi(\vec{b})}$ must contain such a loop as well. This implies that $|\phi(\vec{a})| = |\phi(\vec{b})|$ holds. In fact, by \eqref{E:quasihom} we have a number $\lambda_\phi > 0$ such that
\[ |\phi(\vec{a})| = \lambda_\phi|\vec{a}| \quad (\vec{a}\in E). \]

Finally, let $\vec{a}, \vec{b}\in E$ be two orthogonal unit vector. On the one hand we have 
\[
1 = \sqrt{\Area(\vec{a},\vec{b})} = \sqrt{\Area(\phi(\vec{a}),\phi(\vec{b}))} \leq |\phi(\vec{a})|.
\]
On the other hand, since \eqref{E:quasihom} holds and $g_\phi$ is a homeomorphism, we get that $\ran\phi \cap \{\vec{v},-\vec{v}\} \neq \emptyset$ for every $\vec{v}\in E$. Therefore there exists a unit vector $\vec{c}\in E$ such that $\langle\phi(\vec{a}),\phi(\vec{c})\rangle = 0$. Then we have
\[
|\phi(\vec{a})|  = \sqrt{\Area(\phi(\vec{a}),\phi(\vec{c}))} = \sqrt{\Area(\vec{a},\vec{c})} \leq 1.
\]
We conclude that $\lambda_\phi = 1$. But then by \eqref{E:Pardef} and \eqref{E:Parprop} we obtain
\[ |\langle\phi(\vec{a}),\phi(\vec{b})\rangle| = |\langle\vec{a},\vec{b}\rangle| \quad (\vec{a},\vec{b}\in E). \]
Applying the real version of Wigner's theorem we easily complete our proof.
\end{proof}


Next, we prove Theorem \ref{T:MTgen} in two parts, first in the two dimensional case and then in general. As was mentioned before, the verification in two dimensions is based on the three-dimensional version of Theorem \ref{T:RW}.

\begin{proof}[Proof of Theorem \ref{T:MTgen} in two dimensions]
First of all, let us point out that for every $A,B\in \irB_s(\C^2)$ the matrix $[A,B]$ is skew-Hermitian (thus normal) and $\tr[A,B] = 0$. Therefore the singular values of $[A,B]$ coincide and they are equal to $\sqrt{\det[A,B]}$ (where $\det[A,B]\geq 0$). This implies that \eqref{E:PMTUI} is equivalent to the following:
\begin{equation}\label{E:PMTdet}
\det[A,B] = \det[\phi(A),\phi(B)] \qquad (A, B \in \irB_s(\C^2)).
\end{equation}

Let $Z_2 := \{A\in \irB_s(\C^2)\colon \tr A = 0\}$ and define the mapping
\[ \tilde\phi\colon \irB_s(\C^2)\to Z_2\subseteq \irB_s(\C^2), \quad \tilde{\phi}(A) = \phi(A) - \frac{\tr \phi(A)}{2} \cdot I \]
which clearly satisfies \eqref{E:PMTdet}. Since $\tilde{\phi}$ preserves commutativity in both directions, $\ran\tilde{\phi}$ cannot be commutative. Therefore there exist two matrices in $\ran\tilde{\phi}$ which do not commute. We conclude that if $C\in\ran\tilde{\phi}$ commutes with every element of $\ran\tilde{\phi}\subseteq Z_2$, then $C$ commutes with two non-commuting rank-one projections, whence we obtain $C = 0$. This implies that $\tilde{\phi}(A) = 0$ holds if and only if $A = \lambda I$ with some $\lambda\in\R$.

Next, we consider two numbers: $a,b\in\R, a\neq b$ and a unitary matrix $U$. There exist a $c_U(a,b) \neq 0$ and a unitary matrix $V_U$ such that
\[ \tilde{\phi}\left(U
\left(
\begin{matrix}
a & 0 \\
0 & b
\end{matrix}
\right)
U^*
\right) = V_U
\left(
\begin{matrix}
c_U(a,b) & 0 \\
0 & -c_U(a,b)
\end{matrix}
\right)
V_U^*.
\]
If $a$ and $b$ varies but $U$ does not, then by the preservation of commutativity neither does $V_U$. Thus indeed, $V_U$ does not depend on $a$ and $b$. We can write the following where $A\in \irB_s(\C^2)$ does not commute with $U
\left(
\begin{matrix}
a & 0 \\
0 & b
\end{matrix}
\right)
U^*$:
\[ 
0 \neq \det\left[ V_U \left(\begin{matrix} c_U(a,b) & 0 \\ 0 & -c_U(a,b) \end{matrix}\right) V_U^* , \tilde\phi(A)\right] = \det\left[ U \left(\begin{matrix} a & 0 \\ 0 & b \end{matrix}\right) U^* , A\right]
\]
\[ 
= \det\left[ U \left(\begin{matrix} a+t & 0 \\ 0 & b+t \end{matrix}\right) U^* , A\right]
\]
\[
= \det\left[ V_U \left(\begin{matrix} c_U(a+t,b+t) & 0 \\ 0 & -c_U(a+t,b+t) \end{matrix}\right) V_U^* , \tilde\phi(A)\right] \qquad (t\in\R).
\]
From this equation we immediately obtain 
\begin{equation}\label{E:quasitrans}
|c_U(a+t,b+t)| = |c_U(a,b)| \qquad (a,b,t\in\R).
\end{equation}

We define the following mapping:
\[
\psi := \tilde{\phi}|Z_2\colon Z_2\to Z_2.
\]
Obviously, by \eqref{E:quasitrans} we conclude 
\begin{equation}\label{E:psi-phi}
\tilde{\phi}(A) = \pm\psi\left(A-\frac{\tr A}{2}I\right) \qquad (A\in \irB_s(\C^2)).
\end{equation}
Now, we identify elements of $Z_2$ with vectors of $\R^3$ using the vector space isomorphism
\[ 
\iota\colon \R^3 \to Z_2, \qquad (a,b,c) \mapsto \left(\begin{matrix} a & b+ic \\ b-ic & -a \end{matrix}\right), 
\]
and we define the following transformation:
\[ 
\xi\colon\R^3\to\R^3, \qquad \xi = \iota^{-1}\circ\psi\circ\iota.
\]
The next two equation-chains show that $\xi$ preserves the norm of the crossproduct (here denoted by $\times$):
\[ 
\det[\iota(a_1,b_1,c_1),\iota(a_2,b_2,c_2)]
\]
\[ 
= \det\left[ \left(\begin{matrix} a_1 & b_1+ic_1 \\ b_1-ic_1 & -a_1 \end{matrix}\right), \left(\begin{matrix} a_2 & b_2+ic_2 \\ b_2-ic_2 & -a_2 \end{matrix}\right)\right]
\]
\[ 
= \det\left( 
\begin{array}{cc}
 2 i \left(b_2 c_1-b_1 c_2\right) & -2 a_2 \left(b_1+i c_1\right)+2 a_1 \left(b_2+i c_2\right) \\
 2 a_2 \left(b_1-i c_1\right)-2 a_1 \left(b_2-i c_2\right) & 2 i \left(-b_2 c_1+b_1 c_2\right)
\end{array}
\right)
\]
\[
= 4 \left((a_2 b_1-a_1 b_2)^2+(b_2 c_1-b_1 c_2)^2+(c_2 a_1-c_1 a_2)^2\right) \]
\[ 
= 4|(a_1,b_1,c_1)\times(a_2,b_2,c_2)|^2 = 4\Area\big((a_1,b_1,c_1);(a_2,b_2,c_2)\big)^2.
\]
and
\[ 
\big|\xi(a_1,b_1,c_1)\times\xi(a_2,b_2,c_2)\big| = 
\]
\[ 
\big|(\iota^{-1}\circ\psi\circ\iota)(a_1,b_1,c_1)\times(\iota^{-1}\circ\psi\circ\iota)(a_2,b_2,c_2)\big| 
\]
\[
= \frac{1}{2} \sqrt{\det\big[(\psi\circ\iota)(a_1,b_1,c_1),(\psi\circ\iota)(a_2,b_2,c_2)\big]}
\]
\[
= \frac{1}{2} \sqrt{\det\big[(\tilde\phi\circ\iota)(a_1,b_1,c_1),(\tilde\phi\circ\iota)(a_2,b_2,c_2)\big]}
\]
\[
= \frac{1}{2} \sqrt{\det\big[\iota(a_1,b_1,c_1),\iota(a_2,b_2,c_2)\big]}
\]
\[
= |(a_1,b_1,c_1)\times(a_2,b_2,c_2)|.
\]
By Theorem \ref{T:RW}, we infer that there exists a function $\epsilon\colon\R^3\to\{-1,1\}$ and an orthogonal, linear operator $R\colon\R^3\to\R^3$ such that
\[
\xi(\vec{v}) = \epsilon(\vec{v})R\vec{v} \qquad (\vec{v}\in\R^3).
\]
Using the linearity of $\iota$, we get that
\[
\eta\colon Z_2 \to Z_2, \quad \eta(C) = \iota\big((\epsilon\circ\iota^{-1})(C)\cdot(\xi\circ\iota^{-1})(C)\big) = (\epsilon\circ\iota^{-1})(C)\cdot\psi(C)
\]
is linear and satisfies
\[ \det[C,D] = \det[\eta(C),\eta(D)] \qquad (C,D\in Z_2). \]
By \eqref{E:psi-phi} we obtain that there exists a function $\tau\colon \irB_s(\C^2)\to\{-1,1\}$ such that $\tau(\cdot)\tilde{\phi}(\cdot)$ is linear and satisfies \eqref{E:PMTdet}. Applying \cite[Theorem 1]{GN}, we immediately conclude that 
\[ \tilde\phi(A) = \tau(A)(UAU^*+g(A)I) \quad (A\in\irB_s(\C^2)) \]
holds with a unitary or antiunitary operator $U$ and a linear functional $g\colon \irB_s(\C^2)\to\R$. Transforming back to the original $\phi$, we easily complete this proof.
\end{proof}

We note that instead of using \cite[Theorem 1]{GN}, we could have computed straightforwardly. However, it would have been quite long, hence we decided to choose the above presented shorter way. Finally, we give the proof of Theorem \ref{T:MTgen} in higher dimensions. The spectrum of an operator $T$ will be denoted by $\sigma(T)$.

\begin{proof}[Proof of Theorem \ref{T:MTgen} in at least three dimensions]
By \cite[Theorem 1.2]{Se} and \cite[Corollary 2]{MS}, we have a unitary or an antiunitary operator $U$ on $\C^d$ and for every $A\in \irB_s(\irH)$ we have a bounded Borel function $f_A$ such that
\[ \phi(A) = Uf_A(A)U^* \qquad (A\in \irB_s(\irH)). \]
Using the functional model of normal operators, we easily obtain $|||N||| = |||KNK|||$ where $K$ denotes the coordinate-wise conjugation antilinear operator with respect to some orthonormal base. Therefore the mapping 
\[
\psi\colon \irB_s(\irH)\to \irB_s(\irH), \quad \psi(A) = f_A(A) = U^*\phi(A)U
\] 
obviously satisfies \eqref{E:PMTUI}.

Let us consider an orthogonal decomposition $\irH = \irM\oplus\irM^\perp$ where $\dim\irM = 2$, and the following set:
\[
\irB_s^\irM(\irH) := \{A\in \irB_s(\irH) \colon \irM\in\Red A, A|\irM^\perp = 0\}
\]
where $\Red A$ denotes the set of all reducing subspaces of $A$ (i.~e.~the set of those $A$-invariant subspaces which are also $A^*$-invariant). It is quite easy to see that $\psi(\irB_s^\irM(\irH)) \subseteq \irB_s^\irM(\irH) + \R\cdot I$ holds. We also have
\[ 
|||[\tilde{A}\oplus 0, \tilde{B}\oplus 0]||| = |||[\psi(\tilde{A}\oplus 0), \psi(\tilde{B}\oplus 0)]||| 
\] 
\[ = |||[f_{\tilde{A}\oplus 0}(\tilde{A})\oplus f_{\tilde{A}\oplus 0}(0), f_{\tilde{B}\oplus 0}(\tilde{B})\oplus f_{\tilde{B}\oplus 0}(0)]||| 
\]
\[ 
= |||[f_{\tilde{A}\oplus 0}(\tilde{A})\oplus 0, f_{\tilde{B}\oplus 0}(\tilde{B})\oplus 0]|||, \quad (\tilde{A}, \tilde{B}\in \irB_s(\irM)). 
\]
Let $\irP_1(\irH)$ denote the set of self-adjoint and rank-one idempotents. If the unit vector $x\in\irH$ lies in the range of $P\in\irP_1(\irH)$, we will use the notation $P = x\otimes x$. Since $|||\cdot \oplus 0|||$ obviously defines a unitarily invariant norm on $\irB(\irM)$, applying the two-dimensional version of Theorem \ref{T:MTgen}, we get 
\[
\psi(P) = f_P(P) \in \{P+\R I, -P+\R I\} \quad (P\in\irP_1(\irH)).
\] 
We define
\[
\tilde\psi\colon \irB_s(\irH)\to \irB_s(\irH), \quad \tilde\psi(A) = \tilde\tau(A)\cdot\psi(A) + \tilde{f}(A)I = g_A(A)
\]
with two arbitrary functions $\tilde\tau\colon \irB_s(\irH)\to\{-1,1\}$, $\tilde f\colon \irB_s(\irH)\to\R$, and for every $A\in \irB_s(\irH)$ a Borel function $g_A$. Obviously $\tilde{\psi}$ satisfies \eqref{E:PMTUI} as well, moreover, we may suppose that $\tilde{\psi}(P) = P$ holds for every $P\in\irP_1(\irH)$.

Finally, we will use the following equation:
\begin{equation}\label{E:A-P}
\big|\big|\big|[A,x\otimes x]\big|\big|\big| = c\sqrt{\langle A^2x,x\rangle - \langle Ax,x\rangle^2} \quad (A\in \irB_s(\irH), x\in\irH, \|x\| = 1)
\end{equation}
with a number $c>0$. A proof for \eqref{E:A-P} was provided in \cite[p. 461]{N} and as was mentioned in \cite[p. 3862]{MT} the same method can be applied for every self-adjoint operator. For any unit vector $x\in\irH$ and $A\in\irB_s(\irH)$ by \eqref{E:PMTUI} and \eqref{E:A-P}, we have
\[
\langle A^2x,x\rangle - \langle Ax,x\rangle^2 = \langle g_A(A)^2x,x\rangle - \langle g_A(A)x,x\rangle^2.
\]
Applying \cite[Proposition]{MT}, $\{A - g_A(A),A + g_A(A)\}\cap(\R\cdot I) \neq \emptyset$ follows immediately. Transforming back to our original $\phi$, we easily complete our proof.
\end{proof}


\section{Remarks and open problems}

This section is devoted to giving some remarks and posing some open problems. The following theorem was proven in \cite{GN} directly. However, we would like to point out that there is another way to verify it, namely by extending the mapping (which acts on $\irP_1(\C^2)$) to a map on $\irB_s(\C^2)$ and using Theorem \ref{T:MTgen}.

\begin{corollary}[Theorem 2 of \cite{GN}]
Assume that $|||\cdot|||$ is an arbitrary unitarily invariant norm. Let $\Phi\colon \irP_1(\C^2)\to \irP_1(\C^2)$ be a map for which
\begin{equation}\label{E:nocp}
\big|\big|\big|[\Phi(P),\Phi(Q)]\big|\big|\big|=\big|\big|\big|[P,Q]\big|\big|\big|\quad(P,Q\in \irP_1(\C^2))
\end{equation}
is satisfied. Then there exists a unitary or an antiunitary operator $U$ on $\C^2$ such that for each $P\in \irP_1(\C^2)$ we have 
\begin{equation}\label{E:projchoice}
\Phi(P) \in \{UPU^*, UP^{\bot}U^*\}.
\end{equation}
\end{corollary}

\begin{proof}
By the observations following the proof of Claim 3 in the proof of \cite[Theorem 2]{GN}, we may suppose that $\Phi$ is injective. Thus $I-\Phi(P) = \Phi(I-P)$ holds for every $P\in\irP_1(\C^2)$. Now, let us define the transformation:
\[
\phi\colon \irB_s(\C^2)\to\irB_s(\C^2), 
\]
\[
\phi(\lambda P + \mu(I-P)) = \lambda \Phi(P) + \mu\Phi(I-P) \quad (\lambda,\mu\in\R, P\in \irP_1(\C^2)).
\]
Clearly, $\phi$ is well-defined. The equation
\[
|||[\lambda P + \mu(I-P),\lambda' Q + \mu'(I-Q)]||| = |\lambda-\mu|\cdot|\lambda'-\mu'|\cdot|||[P,Q]|||
\]
\[ = |\lambda-\mu|\cdot|\lambda'-\mu'|\cdot|||[\Phi(P),\Phi(Q)]||| 
\]
\[
= |||[\lambda \Phi(P) + \mu\Phi(I-P),\lambda' \Phi(Q) + \mu'\Phi(I-Q)]|||
\]
is satisfied for every $P, Q\in \irP_1(\C^2)$ and $\lambda,\lambda',\mu,\mu'\in\R$. Using Theorem \ref{T:MTgen}, we obtain \eqref{E:projchoice}.
\end{proof}

Next, we intend to make some notes on the famous Uhlhorn theorem in finite dimensions. U.~Uhlhorn proved in \cite{U} that if $\dim\irH\geq 3$, then every bijective mapping $\phi\colon\irP_1(\irH)\to\irP_1(\irH)$ which preserves orthogonality in both directions is induced by a unitary or an antiunitary operator. Clearly, for different rank-one projections $P$ and $Q$, they commute if and only if they are orthogonal. If we drop the bijectivity condition in Uhlhorn's theorem, then, in general, a similar conclusion with linear or antilinear isometries (as in the non-bijective version of Wigner's theorem) does not hold. A counterexample was provided e.~g.~in \cite{Se2} or in \cite[Section 3]{GN}. On the other hand, using \v Semrl's result \cite[Theorem 1.2]{Se} and a similar extension technique as in the above Corollary, we can easily verify the following.

\begin{corollary}
Let $d\geq 3$ and $\Phi\colon\irP_1(\C^d)\to\irP_1(\C^d)$ be a (not necessarily bijective) transformation which preserves orthogonality in both directions. Then there exists a unitary or an antiunitary operator $U$ on $\C^d$ such that
\[ \Phi(P) = UPU^* \qquad (P\in\irP_1(\C^d)). \]
\end{corollary}

However, using the non-surjective version of the fundamental theorem of projective geometry, A. Fo\v sner, B. Kuzma, T. Kuzma and N.-S. Sze showed the following, stronger result.

\begin{theorem}[A. Fo\v sner, B. Kuzma, T. Kuzma and N.-S. Sze, \cite{FKKS}, 2011]
Let $d\geq 3$ and $\Phi\colon\irP_1(\C^d)\to\irP_1(\C^d)$ be an arbitrary transformation which preserves orthogonality in one direction (nothing else is assumed), i.~e.~ $\Phi(P)\perp\Phi(Q)$ holds whenever $P\perp Q$. Then there exists a unitary or an antiunitary operator $U$ on $\C^d$ such that
\[ \Phi(P) = UPU^* \qquad (P\in\irP_1(\C^d)). \]
\end{theorem}

We close this article with posing some open problems. In case when $\dim E = \infty$, throughout the verification of Theorem \ref{T:RW} bijectivity was crucial. We do not know what happens if we drop this condition and it seems to be a really hard question. It is also a very natural question what happens if in Theorem \ref{T:RW} we only assume that $\Area(\phi(\vec{a}),\phi(\vec{b})) = 1$ holds if and only if $\Area(\vec{a},\vec{b}) = 1$. A weaker version of this latter problem is if we demand that $\Area(\phi(\vec{a}),\phi(\vec{b})) = 1$ holds if $\Area(\vec{a},\vec{b}) = 1$. We were not able to find any example which has the latter property but which does not satisfy \eqref{E:Parprop}. If the same conclusion held, it would be an Uhlhorn-type generalization of our result.


\section*{Acknowledgements.}
The author emphasizes his thanks to Dr.~J\'anos Kincses for many consultations about algebraic topology.

The author was supported by the "Lend\" ulet" Program (LP2012-46/2012) of the Hungarian Academy of Sciences.

\bibliographystyle{amsplain}

\end{document}